\documentclass[12pt]{amsart}\usepackage{amsmath,amsthm,amsfonts,amscd,amssymb,eucal,latexsym,mathrsfs,enumerate,framed,tocvsec2}

\usepackage[all]{xy}

\CompileMatrices 

\usepackage[backref=page]{hyperref}

\numberwithin{equation}{section}

\theoremstyle{plain}
\newtheorem{theorem}{Theorem}
\newtheorem{lemma}[theorem]{Lemma}
\newtheorem{corollary}[theorem]{Corollary}
\newtheorem{proposition}[theorem]{Proposition}

\theoremstyle{definition}

\newtheorem{definition}[theorem]{Definition}
\newtheorem{remark}[theorem]{Remark}

\newcommand{\Hs}{\mathcal H}
\newcommand{\Ks}{\mathcal K}

\newcommand{\cb}{\mathit{cb}}

\newcommand{\N}{\mathcal N} 
\newcommand{\eps}{\varepsilon}

\newcommand{\vp}{\varepsilon}

\newcommand{\puk}{\mathrm{Puk}}

\newcommand{\B}{\mathcal B}

\newcommand{\vnotimes}{\,\overline{\otimes}\,}

\settocdepth{section}

\numberwithin{theorem}{section}

\begin{document}

\title{Structural properties of close II$_1$ factors}

\author[Cameron et. al.]{Jan Cameron}
\address{\hskip-\parindent
Jan Cameron, Department of Mathematics, Vassar College, Poughkeepsie, NY 12604, 
U.S.A.}
\email{jacameron@vassar.edu}

\author[]{Erik Christensen}
\address{\hskip-\parindent
Erik Christensen, Institute for Mathematiske Fag, University of Copenhagen,
Copenhagen, Denmark.}
\email{echris@math.ku.dk}
\author[]{Allan M.~Sinclair}
\address{\hskip-\parindent
Allan M.~Sinclair, School of Mathematics, University of Edinburgh, JCMB, King's
Buildings, Mayfield Road, Edinburgh, EH9 3JZ, Scotland.}
\email{allan.m.sinclair@gmail.com}
\author[]{Roger R.~Smith}
\address{\hskip-\parindent
Roger R.~Smith, Department of Mathematics, Texas A{\&}M University,
College Station, TX 77843,  U.S.A.}
\email{rsmith@math.tamu.edu}

\author[]{Stuart White}
\address{\hskip-\parindent
Stuart White, School of Mathematics and Statistics, University of Glasgow, 
University Gardens, Glasgow Q12 8QW, Scotland.}
\email{stuart.white@glasgow.ac.uk}

\author[]{Alan D.~Wiggins}
\address{\hskip-\parindent
Alan D.~Wiggins, Department of Mathematics and Statistics, University of
Michigan-Dearborn,  Dearborn, MI 48126, U.S.A.}
\email{adwiggin@umd.umich.edu}
\thanks{The research leading to this paper was supported by an AMS-Simons research travel grant (JC), NSF grant DMS-1101403 (RRS), and EPSRC grant EP/IO19227/1 (SW)}

\begin{abstract}
We show that a number of key structural properties transfer between sufficiently close II$_1$ factors, including solidity, strong solidity, uniqueness of Cartan masas and property $\Gamma$. We also examine II$_1$ factors close to tensor product factors, showing that such factors also factorise as a tensor product in a fashion close to the original.
\end{abstract}

\maketitle

\section{Introduction}

In \cite{KK:AJM}, Kadison and Kastler equipped the collection of all operator algebras acting on a Hilbert space with a metric which measures how close the unit balls of two algebras are in operator norm.  Using the operator norm in this fashion makes closeness a very strong condition on a pair of operator algebras, leading Kadison and Kastler to conjecture that sufficiently close algebras should be spatially isomorphic.  Strong results for amenable von Neumann algebras were obtained in the late 1970's in \cite{C:Invent,RT:JFA,C:Acta}: sufficiently close amenable von Neumann algebras must arise from small unitary perturbations. A few years ago corresponding  results for separable nuclear $\mathrm{C}^*$-algebras were obtained in \cite{CSSWW:Acta} (examples of Johnson from \cite{Joh:CMB} show that one can only expect a small unitary perturbation in the point norm topology in the $\mathrm{C}^*$-setting). In \cite{CCSSWW:Duke} we examined nonamenable algebras, providing the first nonamenable von Neumann algebras which satisfy the Kadison-Kastler conjecture (an expository account of this work can be found in \cite{CCSSWW:PNAS}).

The driving theme of this paper is the transfer of structural properties between close von Neumann algebras. This was the focus of the original paper \cite{KK:AJM}, which shows that close von Neumann algebras $M$ and $N$ have the same nonzero summands in their type decomposition, and further that the corresponding summands are again close.  Subsequently close $\mathrm{C}^*$-algebras were shown to have isomorphic ideal lattices (and correspondingly close ideals) by Phillips in \cite{P:IUMJ}, and $\mathrm{C}^*$-algebras which remain close under all matrix amplifications were shown to have isomorphic $K$-theories by Khoshkam in \cite{Kh:JOT}. Recently questions of this nature have been explored for more refined $\mathrm{C}^*$-algebra invariants in \cite{CSSW:GAFA} (which demonstrates a strong connection between close operator algebras and Kadison's similarity problem from \cite{K:AJM}, which the authors extended in \cite{CCSSWW:InPrep}) and \cite{PTWW:APDE}.  In this paper we turn to von Neumann algebras, and more precisely II$_1$ factors, showing how the methods developed in \cite{CCSSWW:Duke} can be used to examine properties such as (strong) solidity \cite{O:Acta,OP:Ann} and uniqueness of Cartan masas \cite{OP:Ann}, which have come to the forefront as part of the revolutionary progress in the structure theory of II$_1$ factors  made over the last fifteen years.  We also consider Murray and von Neumann's property $\Gamma$ and tensorial decompositions, transferring these properties to sufficiently close factors, and examine the structure of masas within close factors.

Before proceeding, we recall the definitions of the Kadison-Kastler metric and the closely related notion of near containments from \cite{KK:AJM} and \cite{C:Acta} respectively.  Note that the metric is not quite obtained from symmetrising the notion of near inclusion.

\begin{definition}[Kadison-Kastler, Christensen]
Let $M$ and $N$ be von Neumann algebras acting nondegenerately on a Hilbert space $\Hs$.  The distance, $d(M,N)$ is the infimum of those $\gamma>0$ such that for every operator $x$ in the unit ball of one algebra, there exists $y$ in the unit ball of the other algebra with $\|x-y\|<\gamma$.  A near containment $M\subseteq_\gamma N$ arises when for every $x\in M$, there exists $y\in N$ with $\|x-y\|\leq\gamma\|x\|$.  Write $M\subset_\gamma N$ when there exists $\gamma'<\gamma$ with $M\subseteq_{\gamma'} N$. 
\end{definition}

We note that there is no assumption that $\|y\|\leq \|x\|$ in the definition of a near containment. Consequently, the composition of near containments $P\subset_\alpha Q \subset_\beta R$ becomes 
\begin{equation}\label{containeq}
P\subset_{(\alpha +\beta +\alpha \beta)} R,
\end{equation}
easily obtained from the triangle inequality.

It is also natural to consider `completely bounded' versions of the above notions.  Let $d_{\cb}(M,N)=\sup_nd(M\otimes\mathbb M_n,N\otimes\mathbb M_n)$, where one measures the distance between $M\otimes\mathbb M_n$ and $N\otimes\mathbb M_n$ on $\Hs\otimes\mathbb C^n$.  Similarly, write $M\subset_{\cb,\gamma}N$ when $M\otimes\mathbb M_n\subset_{\gamma}N\otimes\mathbb M_n$ for all $n\in\mathbb N$.

A key tool in the study of close von Neumann algebras is the embedding theorem for a near containment of an amenable von Neumann algebra from \cite[Theorem 4.3, Corollary 4.4]{C:Acta}. This is used repeatedly in this paper and so we recall the statement here for the reader's convenience.

\begin{theorem}[Christensen]\label{Injective} Let $P$ be an amenable von Neumann subalgebra of $\mathcal B(\Hs)$ containing
$I_\Hs$. Suppose that
$B$ is another von Neumann subalgebra of $\mathcal B(\Hs)$ and $P\subset_\gamma
B$ for a constant $\gamma<1/100$.  Then there exists a unitary $u\in (P\cup
B)''$ with $uPu^*\subseteq B$, $\|I_\Hs-u\|\leq150\gamma$ and
$\|uxu^*-x\|\leq100\gamma\|x\|$
for $x\in P$. If, in addition, $\gamma<1/101$ and $B\subset_\gamma P$, then
$uPu^*=B$.
\end{theorem}

In the next section we consider the structure of close masas, providing a one-to-one correspondence between unitary equivalence classes of Cartan masas, and transfer property $\Gamma$, solidity and strong solidity to close factors. These results were originally given in the preprint version of \cite{CCSSWW:Duke} on the arXiv, but were removed from the publication version. In Section \ref{S3} we consider tensor product decompositions, and the paper ends with a short list of open problems in Section \ref{S4}.

\section{Masas, solidity and property $\Gamma$}
We start with  the structure of maximal abelian subalgebras (masas) in close II$_1$ factors. Recall that in \cite{D:Ann} Diximer introduced a rough classification of masas $A$ in a II$_1$ factor $M$ through their \emph{normalisers}, namely those unitaries $u\in M$ with $uAu^*=A$.  The collection of all normalisers is denoted $\mathcal N(A\subseteq M)$ and $A$ is said to be \emph{Cartan} when $\mathcal N(A\subseteq M)$ generates $M$ as a von Neumann algebra (general subalgebras $P$ of $M$ with $\mathcal N(P\subseteq M)''=M$ are called regular). At the other extreme, $A$ is said to be \emph{singular} when $\mathcal N(A\subseteq M)\subseteq A$. The transfer of normalisers between close pairs of inclusions provided a key tool in \cite{CCSSWW:Duke}, which we use here to examine the behaviour of close masas in close algebras.

Since the breakthrough paper \cite{OP:Ann}, there has been considerable interest in how many Cartan masas a   II$_1$ factor contains, up to unitary conjugacy:
\cite{OP:Ann} gives the first class of factors with a unique Cartan
masa up to unitary conjugacy, \cite{CJ:BAMS} provides the first examples of factors with two Cartan
masas which are not even conjugate by an automorphism, and  \cite{OP:AJM}  presents more
examples of factors with unique Cartan masas and also new factors with  at least
two Cartan masas. More recently, large classes of crossed products have been shown
to have unique Cartan masas \cite{CS:ASENS,PV:Acta,PV:Crelle}. At the other end
of the spectrum,
\cite{SV:arXiv} provides a   II$_1$ factor with unclassifiably many Cartan
masas up to conjugacy by an automorphism. Our first objective is to show that close   II$_1$
factors have the same Cartan masa structure.  Given a  II$_1$
factor $M$, let $\mathrm{Cartan}(M)$ be the collection of equivalence classes of
Cartan masas in $M$ under the relation $A_1\sim A_2$ if and only if there is a
unitary $u\in M$ with $uA_1u^*=A_2$.

\begin{theorem}\label{CartanThm}
Let $M$ and $N$ be   {\rm{II}}$_1$ factors with separable preduals acting nondegenerately on a
Hilbert space $\Hs$
with
$M\subset_\gamma N$ and $N\subset_\gamma M$ for a constant $\gamma<5.7\times
10^{-16}$.  
\begin{enumerate}[(i)]
\item\label{CartanThm:Part1} Suppose $P\subseteq M$ is an amenable regular von
Neumann subalgebra with $P'\cap M\subseteq P$ and  $Q\subseteq N$ is a von
Neumann subalgebra with 
$P\subset_\delta Q\subset_\delta P$ for some $\delta\geq 0$ such that $300\gamma+\delta<1/8$. Then $Q$ is regular in $N$ and satisfies $Q'\cap N\subseteq Q$.
\item\label{CartanThm:Part2} If $A$ is a Cartan masa in $M$, then there exists a Cartan masa
$B$ in $N$ satisfying
$d(A,B)<100\gamma$.
\item\label{CartanThm:Part3} There exists a canonical bijective map $\Theta :
{\mathrm{Cartan}}(M)\to {\mathrm{Cartan}}(N)$, which is defined by $\theta([A])=[B]$ where
$A\subseteq M$ and $B\subseteq N$ are Cartan masas with $d(A,B)<100\gamma$.
\item\label{CartanThm:Part4} If $M$ has a unique Cartan masa up to unitary conjugacy, then
the same is true for $N$.
\end{enumerate}
\end{theorem}
\begin{proof}
(\ref{CartanThm:Part1}). Since $\gamma<1/100$, we may apply the embedding theorem (Theorem \ref{Injective}) to obtain a
unitary $u\in (P\cup N)''$
satisfying $\|u-I_\Hs\|\leq 150\gamma$, $uPu^*\subseteq N$, and $d(P,uPu^*)\leq
100\gamma$. Define $N_1=u^*Nu$, so
that $P\subseteq M\cap N_1$ and $M\subset_{\gamma_1}N_1\subset_{\gamma_1}M$
where $\gamma_1=301\gamma$. Then
the bound on $\gamma$ gives $\gamma_1<1.74\times 10^{-13}$, so we may apply
\cite[Lemma 4.10]{CCSSWW:Duke}
 to conclude
that $P$ is regular in $N_1$ and $P'\cap N_1\subseteq P$. Thus $Q_1:=uPu^*$ is
regular in $N$ and satisfies
$Q_1'\cap N\subseteq Q_1$. Now by \cite[Equation (2.1)]{CCSSWW:Duke},
$Q_1\subset_{\eta}Q\subset_{\eta}Q_1$ where
$\eta=300\gamma +\delta <1/8$.  By \cite[Theorem 4.1]{C:IJM}, $Q$ and $Q_1$
are unitarily conjugate inside
$N$ (strictly speaking, the statement of \cite[Theorem 4.1]{C:IJM} requires the hypothesis $d(Q,Q_1)<\frac{1}{8}$ but, as noted in \cite[Section 3]{CCSSWW:Duke}, the proof only needs the hypothesis in terms of near inclusions). Thus $Q$ inherits the desired properties from $Q_1$.

(\ref{CartanThm:Part2}). Given a Cartan masa $A$ in $M$, Theorem 1.2 gives a unitary $u\in (A\cup N)''$ such that the algebra $B:=uAu^*$ lies in $N$ and satisfies $d(A,B)<100\gamma$. Then $B$ is a masa in $N$ by \cite[Lemma 2.17]{CCSSWW:Duke}
 and so is Cartan by (\ref{CartanThm:Part1}).

(\ref{CartanThm:Part3}). From (\ref{CartanThm:Part2}), we may associate to each Cartan masa $A$ in $M$  a
Cartan masa $B$ in $N$ so that
$d(A,B)< 100\gamma$. Let $A_1$ be another Cartan masa in $M$ and choose a Cartan
masa $B_1$ in $N$ with
$d(A_1,B_1)< 100\gamma$. If there exists a unitary $u\in M$ such that
$A_1=uAu^*$, then by \cite[Lemma 2.12 (i)]{CCSSWW:Duke},
 there is a unitary $v\in N$ with $\|u-v\|<\sqrt{2}\gamma$. Then
\begin{align}
d(B_1,vBv^*)&\leq d(B_1,uBu^*) +2\|u-v\|< d(B_1,uAu^*)+ 2\sqrt{2}\gamma
+100\gamma\notag\\
&=d(B_1,A_1) +(100+2\sqrt{2})\gamma<(200 +2\sqrt{2})\gamma <1/8.
\end{align}
Thus $B_1$ and $vBv^*$ are unitarily conjugate in $N$ by \cite[Theorem 4.1]{C:IJM}, and hence $B_1$ and $B$ are
unitarily conjugate in $N$. This
shows that there is a well defined map $\Theta :{\mathrm{Cartan}}(M)\to
{\mathrm{Cartan}}(N)$, defined on
$[A]$ by choosing a Cartan masa $B$ as above and letting $\Theta([A])=[B]$. In
the same way there is a map
$\Phi: \mathrm{Cartan}(N)\to \mathrm{Cartan}(M)$ so that for each Cartan masa
$B$ in $N$, $\Phi([B])=[A]$
where $A\subseteq M$ is chosen so that $d(B,A)< 100\gamma$. By construction
$\Phi$ is the inverse of $\Theta$ so $\Theta$ is bijective.

(\ref{CartanThm:Part4}). This is immediate from (\ref{CartanThm:Part3}).
\end{proof}

At the other end of the spectrum, one has the singular masas.  Various ad hoc methods have been used to determine whether certain explicit singular masas are conjugate via an automorphism of the underlying factor; perhaps the most successful is Puk\`anszky's invariant, originating in \cite{P:CJM}, which associates to a masa $A\subseteq M$ a nonempty subset of $\mathbb N\cup \{\infty\}$ as follows: the relative commutant of $A$ inside the basic construction algebra $\langle M,e_A\rangle$ gives a type I von Neumann algebra $A'\cap \langle M,e_A\rangle=(A\cup J_MAJ_M)'$.   This always has a type I$_1$ summand, as $e_A$ is central in $A'\cap \langle M,e_A\rangle$ with $e_A(A'\cap \langle M,e_A\rangle)=e_AA$.  The Puk\'anszky invariant $\puk(A\subseteq M)$ consists of those $n\in\mathbb N\cup\{\infty\}$ such that $(1-e_A)(A'\cap \langle M,e_A\rangle)$ has a nonzero type I$_n$ component.  See \cite[Chapter 7]{SS:Book2} for more information on the Puk\'anszky invariant (including proofs of the facts above). In the next result we do not need the precise definitions of the basic construction, just that the Puk\'anszky invariant is obtained from the relative positions of $e_A$, $A$ and $\langle M,e_A\rangle$. Note that the embedding theorem can be used to provide algebras $B\subseteq N$ satisfying the estimates of the next proposition, when $\delta$ is sufficiently small.
\begin{proposition}\label{MasaProp}
Let $M$ and $N$ be   {\rm{II}}$_1$ factors with separable preduals acting nondegenerately on a
Hilbert space $\Hs$
with
$M\subset_\gamma N$ and $N\subset_\gamma M$ for a constant $\gamma$. Let $A\subseteq M$ be a masa in $M$.
\begin{enumerate}[(i)]
\item Suppose that $\delta>0$ satisfies $(4+2\sqrt{2})(\gamma+24\delta)<1$. If $A$ is singular, then any subalgebra $B\subseteq N$ with $d(A,B)<\delta$ is a singular masa in $N$.
\item Suppose that $\delta>0$ satisfies $(\gamma+24\delta)<1.74\times 10^{-13}$.  Then any von Neumann subalgebra $B\subseteq N$ with $d(A,B)<\delta$ is a masa in $N$ satisfying 
\begin{equation}\puk(A\subseteq M)=\puk(B\subseteq N).
\end{equation}
\end{enumerate}
\end{proposition}
\begin{proof}
First note that \cite[Lemma 2.3]{C:JLMS} shows that $B$ is abelian.  Then \cite[Lemma 2.16(i)]{CCSSWW:Duke} gives $B'\cap N
\subset_{2\sqrt{2}\delta+\gamma} A'\cap M=A\subset_{\delta} B$ so that 
$B'\cap N\subset_{\eta}B\subseteq B'\cap N$, where
$\eta=2\sqrt{2}\delta+\gamma+ (1+\gamma+2\sqrt{2}\delta)\delta$ using \eqref{containeq}. Since the hypothesis of (i) implies that $\eta<1$, we have $B=B'\cap N$ (see \cite[Proposition 2.4]{CSSWW:Acta}). Thus $B$ is a masa in $N$.   As both $A$ and $B$ are amenable, by \cite[Corollary 4.2(c)]{C:Acta} there exists a unitary $u\in (A\cup B)''$ with $\|u-1\|<12\delta$ and $uAu^*=B$. Write $N_1=u^*Nu$ so that $A$ is a masa in $N_1$ and the near inclusions $M\subset_{\gamma+24\delta}N_1$ and $N_1\subset_{\gamma+24\delta}M$ hold.

Now suppose $A\subseteq M$ is singular. Given any unitary normaliser $v\in \N(A\subseteq N_1)$, \cite[Lemma 3.4(iii)]{CCSSWW:Duke} provides a unitary normaliser $v'\in \N(A\subseteq M)$ with $\|v-v'\|\leq (4+2\sqrt{2})(\gamma+24\delta)<1$.  By \cite[Proposition 3.2]{CCSSWW:Duke}, we have $v'=vu_1u_2$ for unitaries $u_1\in A$ and $u_2\in A'\cap \mathcal B(\Hs)$.  Thus $vxv^*=v'xv'^*$ for all $x\in A$. Since $A\subseteq M$ is singular, it follows that $vxv^*=x$ for all $x\in A$, and so $v\in A$ since $A$ is a masa in $N_1$. Thus $A$ is singular in $N_1$, and so $B$ is singular in $N$, proving (i).

For (ii), as $M\subset_{\gamma_1}N_1$ and $N_1\subset_{\gamma_1}M$ for $\gamma_1=(\gamma+24\delta)<1.74\times 10^{-13}$, we can use \cite[Lemma 4.10]{CCSSWW:Duke} (with $P=A$) to simultaneously represent $M$ and $N_1$ on a new Hilbert space $\Ks$ such that both these algebras are simultaneously in standard form with respect to the same trace vector, and have equal basic constructions $\langle M,e_A\rangle=\langle N_1,e_A\rangle$.  It follows that $\puk(A\subseteq M)=\puk(A\subseteq N_1)$, and hence $\puk(A\subseteq M)=\puk(B\subseteq N)$.
\end{proof}

Recall from \cite{O:Acta} that a II$_1$ factor is said to be \emph{solid} when
every diffuse  unital von Neumann subalgebra $P\subseteq M$ has an
amenable relative commutant $P'\cap M$.   Note that to establish solidity of $M$ it suffices to show that $P'\cap M$ is amenable when $P$ is diffuse and amenable (or abelian), as given a general diffuse subalgebra $P$ of $M$, take a maximal abelian subalgebra $P_0$ of $P$.  This will be diffuse and $P'\cap M\subseteq P_0'\cap M$ so that $P'\cap M$ will inherit amenability from $P_0'\cap M$ (since $M$ is finite).

Subsequently Ozawa and Popa
generalised the concept of
solidity further in \cite{OP:Ann}: a   II$_1$ factor $M$ is said to be \emph{strongly solid}
if every unital diffuse amenable subalgebra $B\subseteq M$ has an amenable
normalizing algebra $\N(B\subseteq M)''$.  Both these properties transfer to sufficiently close factors, as we now show.

\begin{proposition}\label{SSolid}
Let $M$ and $N$ be   ${\mathrm{II}}_1$ factors acting
nondegenerately on a Hilbert space
$\Hs$ with $d(M,N)<1/3200$.  Then: 
\begin{enumerate}[\rm (i)]
\item\label{SSolid:Part1} $M$
is solid if and only if $N$ is solid;
\item\label{SSolid:Part2} $M$
is strongly solid if and only if $N$ is strongly solid.
\end{enumerate}
\end{proposition}

\begin{proof} 
Let $M$ and $N$ be   ${\mathrm{II}}_1$ factors acting
nondegenerately on a Hilbert space $\Hs$ with $d(M,N)<\gamma<1/3200$. We will assume that $N$ is solid, or strongly solid, and show that $M$ has the same property, so take a diffuse unital amenable subalgebra $P$ of $M$.  By Theorem
\ref{Injective}, there exists a unital von Neumann subalgebra $Q\subseteq N$
isomorphic to $P$ such that $d(P,Q)\leq 100\gamma$.   When $N$ is strongly solid, let $Q_1=\mathcal N(Q\subseteq N)''\subseteq N$, and when $N$ is solid, let $Q_1=(Q\cup (Q'\cap N))''\subseteq N$.  In both cases $Q_1$ is amenable. This is the hypothesis of strong solidity in the first case, while when $N$ is solid, $Q'\cap N$ is amenable, which implies that $Q_1$ is amenable as it is the von Neumann algebra generated by two commuting amenable subalgebras.   (One way to see this is via the equivalence of injectivity and hyperfiniteness, since certainly two commuting finite dimensional algebras generate another finite dimensional algebra).

Applying Theorem \ref{Injective} again gives a unitary $u\in (Q_1\cup M)''$ such that $uQ_1u^*\subseteq M$, $\|u-I_{\Hs}\|<150\gamma$ and $d(uQ_1u^*,Q_1)\leq 100\gamma$.  Thus $d(P,uQu^*)\leq d(P,Q)+2\|u-I_{\Hs}\|\leq 400\gamma$.  Since $400\gamma<1/8$, \cite[Theorem 4.1]{C:IJM} gives a unitary $u_1\in (P\cup uQu^*)''\subseteq M$ satisfying $u_1Pu_1^*=uQu^*$ and $\|u_1-I_\Hs\|\leq 7d(P,uQu^*)\leq 2800\gamma$ (here we have crudely estimated the function $\delta$ appearing in \cite[Theorem 4.1]{C:IJM}).

Now write $N_1=u_1^*uNu^*u_1$ so that $P=u_1^*uQu^*u_1$ is a subalgebra of $N_1$. Since $P'\cap N_1= u_1^*u(Q'\cap N)u^*u_1$,  $Q'\cap N\subseteq Q_1$, and $u_1\in M$, we have $P'\cap N_1\subseteq P'\cap M$.  As $u_1\in M$, 
\begin{align}
d(M,N_1)&=d(u_1Mu_1^*,uNu^*)=d(M,uNu^*)\notag\\
&\leq
d(M,N)+2\|u-I_{\Hs}\|<301\gamma.
\end{align}
By \cite[Lemma 2.16 (i)]{CCSSWW:Duke} (with $\delta =0$) we have $P'\cap M\subseteq_{301\gamma}P'\cap N_1$. So, as $301\gamma<1$,  $P'\cap M=P'\cap N_1$ (this is a folklore Banach
space argument, see
\cite[Proposition 2.4]{CSSWW:Acta} for the precise statement being used). In the case when $N$ is solid, $Q'\cap N$ is amenable, and hence $P'\cap M=P'\cap N_1=u_1^*u(Q'\cap N)u^*u_1$ is amenable.  This proves that $M$ is solid. 

In the case when $N$ is strongly solid, note that 
\begin{equation}
\N(P\subseteq N_1)''= u_1^*u\N(Q\subseteq N)''u^*u_1=u_1^*uQ_1u^*u_1\subseteq M.
\end{equation}
  Now take a unitary $v\in \N(P\subseteq M)$.  As $301\gamma<2^{-3/2}$, \cite[Lemma 3.4 (iii)]{CCSSWW:Duke} provides $v'\in
\N(P\subseteq N_1)\subseteq M$ with $\|v-v'\|\leq (4+2\sqrt{2})301\gamma$. We have
$(4+2\sqrt{2})301\gamma<1$, so
\cite[Proposition 3.2]{CCSSWW:Duke} gives 
unitaries $w\in P$ and $w'\in P'\cap\mathcal B(\Hs)$ satisfying $v'=vww'$.
Then $w'=w^*v^*v'\in P'\cap M$ since $w$, $v$, and $v'$ all lie in $M$. Thus
$w'\in P'\cap N_1\subseteq \N(P\subseteq N_1)''$. Then $v=v'w'^*w^*\in
\N(P\subseteq N_1)$, so that $\N(P\subseteq M)\subseteq \N(P\subseteq
N_1)$. Since $\N(P\subseteq N_1)''$ is amenable so too is its subalgebra $\N(P\subseteq M)''$. Thus $M$ is strongly solid. 
 \end{proof}
 
To conclude this section, we turn to Murray and von Neumann's property $\Gamma$.  Recall that a II$_1$ factor $M$ with trace $\tau$ has property $\Gamma$ if for any finite set $\{x_1,\cdots,x_n\}$ in $M$ and $\eps>0$, there exists a unitary $u\in M$ with $\tau(u)=0$ and $\|[x_i,u]\|_2<\eps$ (as is usual, $\|\cdot\|_2$ denotes the norm induced by the trace: $\|x\|_2=\tau(x^*x)^{1/2}$).  Equivalently (in the presence of a separable predual), property $\Gamma$ is characterised by the nontriviality of the central sequence algebra $M^\omega\cap M'$, where $\omega$ is a free ultrafilter (see \cite[Theorem XIV.4.7]{Tak:Book3}). For II$_1$
factors with nonseparable preduals this equivalence no longer holds (see \cite[Section
3]{FHS:BLMS}) and instead one must work with ultrafilters on sets of larger
cardinality.  For simplicity, we restrict to the separable predual situation
here. However the argument can be modified to handle the nonseparable
situation (with the same constants).  To reach our stability result we need an extension of \cite[Lemma 2.15]{CCSSWW:Duke}.
\begin{lemma}\label{Gamma.L}
Let $M$ and $N$ be   ${\mathrm{II}}_1$ factors represented nondegenerately
on a Hilbert
space $\Hs$ and let $\gamma$ and $\eta$ be positive constants. Suppose that
$d(M,N)<\gamma<1$ and that we have $x_1,x_2$ in the unit ball of
$M$ and $y_1,y_2$ in the unit ball of $N$ with $\|x_i-y_i\|\leq\eta$, $i=1,2$. 
Then
\begin{equation}\label{Gamma.L1}
\|y_1-y_2\|_{2,N}^2\leq\|x_1-x_2\|_{2,M}^2+8\eta+(8\sqrt{2}+8)\gamma.
\end{equation}
\end{lemma}
\begin{proof}
Define $s=y_1-y_2\in N$ and
$t=x_1-x_2\in M$, so that $\|s\|,\,\|t\|\leq 2$ and $\|s-t\|\leq 2\eta$. 
Let $\Phi$ be a state on $\mathcal B(\Hs)$ extending $\tau_M$. 
Then \cite[Lemma 2.15]{CCSSWW:Duke} gives
\begin{equation}
|\tau_N(s^*s)-\Phi(s^*s)|\leq
(2\sqrt{2}+2)\gamma\|s^*s\|\leq(8\sqrt{2}+8)\gamma.
\end{equation}
We also have
\begin{equation}\label{Gamma.L2}
|\Phi(s^*s)-\Phi(t^*t)|\leq\|(s^*-t^*)s+t^*(s-t)\|\leq 8\eta,
\end{equation}
so
\begin{align}
\|s\|^2_{2,N}&=\tau_N(s^*s)\leq |\Phi(s^*s)|+|\tau_N(s^*s)-\Phi(s^*s)|\notag\\
&\leq |\Phi(s^*s)-\Phi(t^*t)|+\Phi(t^*t)+(8\sqrt{2}+8)\gamma\notag\\
&\leq \|t\|^2_{2,M}+8\eta+(8\sqrt{2}+8)\gamma,
\end{align}
since $\Phi$ and $\tau_M$ agree on $M$. This is \eqref{Gamma.L1}. 
\end{proof}

\begin{proposition}\label{GammaProp}
Let $M$ and $N$ be   ${\mathrm{II}}_1$ factors with separable preduals acting
nondegenerately
on a Hilbert space
$\Hs$ with $d(M,N)<\gamma$ for a constant $\gamma<1/190$. Suppose that $M$ has
property
$\Gamma$.  Then $N$ also has property $\Gamma$.
\end{proposition}
\begin{proof}
Suppose that $M$ has property $\Gamma$ and fix a free ultrafilter
$\omega$ on $\mathbb N$. By definition, there is a sequence $(u_n)_{n=1}^\infty$ of trace zero unitaries such that $u=(u_n)$ defines an element in $M^\omega\cap M'$.  For each $n$, use \cite[Lemma 2.12]{CCSSWW:Duke} to find a unitary
$v_n\in N$ with $\|u_n-v_n\|<\sqrt{2}\gamma$ and let $v$ denote the class of
$(v_n)$ in $N^\omega$.  Let $\Phi$ denote a state on ${\mathcal{B}}(\Hs)$
extending
$\tau_N$. Then \cite[Lemma 2.15]{CCSSWW:Duke} gives the estimate
\begin{equation}
|\tau_M(u_n)-\Phi(u_n)|\leq (2\sqrt{2}+2)\gamma,\quad n\in\mathbb N, 
\end{equation}
so that
\begin{equation}
|\tau_M(u_n)-\tau_N(v_n)|\leq |\tau_M(u_n)-\Phi(u_n)|+|\Phi(u_n)-\Phi(v_n)|\leq(
3\sqrt{2}+2)\gamma.
\end{equation}
Thus
\begin{equation}\label{Gamma.E1}
|\tau_{N^\omega}(v)|\leq (3\sqrt{2}+2)\gamma.
\end{equation}

Given a unitary $w\in N$, use \cite[Lemma 2.12]{CCSSWW:Duke} to find a unitary $w'\in M$
with $\|w'-w\|<\sqrt{2}\gamma$. Then 
\begin{equation}
\|w'u_n-wv_n\|\leq \|(w'-w)u_n\|+\|w(u_n-v_n)\|\leq 2\sqrt{2}\gamma
\end{equation}
and similarly $\|u_nw'-v_nw\|\leq 2\sqrt{2}\gamma$. Taking $\eta=2\sqrt{2}\gamma$
in  Lemma \ref{Gamma.L} with $x_1=w'u_n$, $x_2=u_nw'$, $y_1=wv_n$ and $y_2=v_nw$
gives
\begin{equation}
\|wv_n-v_nw\|_{2,N}^2\leq\|w'u_n-u_nw'\|_{2,M}^2+(24\sqrt{2}+8)\gamma.
\end{equation}
Since $\lim_{n\rightarrow\omega}\|w'u_n-u_nw'\|_{2,M}=0$, we have the estimate
\begin{equation}\label{T1.1}
\|wvw^*-v\|_{2,N^\omega}^2=\|wv-vw\|_{2,N^\omega}^2\leq (24\sqrt{2}+8)\gamma
\end{equation}
 in $N^\omega$. Let $y$ be the unique element of minimal $\|\cdot\|_{2,N^\omega}$-norm in
$\overline{{\mathrm{conv}}}^{2,N^\omega}\{wvw^*:w\in\mathcal U(N)\}$. This lies
in $N^\omega$ and uniqueness ensures that $y\in N'\cap N^\omega$. It remains to check
that $y$ is nontrivial.  

The estimate (\ref{T1.1}) gives
\begin{equation}\label{Gamma.E2}
\|y-v\|_{2,N^\omega}^2\leq (24\sqrt{2}+8)\gamma,
\end{equation}
and so
\begin{equation}
\|y\|_{2,N^{\omega}}\geq 1-((24\sqrt{2}+8)\gamma)^{1/2}
\end{equation}
as $\|v\|_{2,N^\omega}=1$. We can estimate
\begin{align}
|\tau_{N^\omega}(y)|&\leq|\tau_{N^\omega}(v)|+|\tau_{N^\omega}(y-v)|
\leq(3\sqrt{2}+2)\gamma+\|y-v\|_{2,N^\omega}\nonumber\\
&\leq(3\sqrt{2}+2)\gamma+((24\sqrt{2}+8)\gamma)^{1/2},
\end{align}
using (\ref{Gamma.E1}), (\ref{Gamma.E2}) and the Cauchy-Schwarz inequality.
If $y\in\mathbb CI_{N^\omega}$, then $y=\tau_{N^\omega}(y)I_{N^\omega}$ so 
$\|y\|_{2,N^\omega}=|\tau_{N^\omega}(y)|$, and it follows that
\begin{equation}
1-((24\sqrt{2}+8)\gamma)^{1/2}\leq
\|y\|_{2,N^{\omega}}\leq(3\sqrt{2}+2)\gamma+((24\sqrt{2}+8)\gamma)^{1/2}.
\end{equation}
Direct computations show that this is a contradiction when $\gamma<1/190$, so
that
$y$ is a nontrivial element of $N'\cap N^\omega$. Therefore $N$ has property $\Gamma$.
\end{proof}

\begin{remark}
As a consequence of the results of this section, factors close to free group factors inherit a number of their properties. Assume $d(M,L\mathbb F_2)$ is sufficiently small. Then $M$ is strongly solid by Proposition \ref{SSolid} and \cite{OP:Ann}, and every masa in $M$ has infinite multiplicity (i.e. unbounded Puk\'anszky invariant) by Proposition \ref{MasaProp} and \cite{D:MA}. Further, there are masas $A_1$ and $A_2$ in $M$ close to the generator masas in $L\mathbb F_2$ and a masa $B$ close to the radial masa in $L\mathbb F_2$.  These are singular with Puk\'anzksy invariant $\{\infty\}$ by Proposition \ref{MasaProp} and \cite{D:Ann,R:PJM,SS:Book2}. The embedding theorem was used in Proposition \ref{SSolid} and can be employed in a similar way to establish maximal injectivity of $A_1$, $A_2$ and $B$ in $M$ since their counterparts in $L\mathbb{F}_2$ are known to be maximal injective  \cite{P:Adv,CFRW:JLMS}.
\end{remark}

\section{Tensor products}\label{S3}

 In \cite[Section 5]{CCSSWW:Duke} we considered McDuff factors (those which absorb the hyperfinite II$_1$ factor tensorially), showing that this property transfers to sufficiently close factors. In this section, we examine general tensor product factorisations, transferring these to close factors. If $P$ and $Q$ are II$_1$ factors, then $M:=P\vnotimes Q$ is generated by two commuting infinite dimensional subalgebras. As shown in \cite{NT:PJA}, this characterises the property of being isomorphic to a tensor product: if $M$ is a II$_1$ factor generated by two commuting infinite dimensional von Neumann subalgebras $S$ and $T$, then $M$ is isomorphic to $S\vnotimes T$, and $S$ and $T$ are automatically II$_1$ factors. This result will prove useful below.

We begin with a technical observation.

\begin{lemma}\label{lemTP.2}
Let $\gamma >0$ and suppose that $M,N\subseteq \mathcal B(\Hs)$ are von Neumann algebras acting nondegenerately on a Hilbert space $\Hs$ such that $d(M,N)< \gamma$. Let $A\subseteq M'\cap N'$ be an abelian von Neumann
algebra. Then
$d((M\cup A)'',(N\cup A)'')<\gamma$.
\end{lemma}
\begin{proof}
Choose $\gamma'$ to satisfy $d(M,N)< \gamma' < \gamma$.
Let $B\subseteq A$ be the span of the projections in $A$, so that $B$ is a $\ast$-subalgebra of $A$. If $x\in {\mathrm{Alg}}(M\cup B)$, $\|x\|\leq 1$, then there exist
orthogonal projections $p_1,\ldots,p_n\in B$ and elements $x_1,\ldots, x_n\in M$ with $\|x_i\|\leq 1$ so that $x=\sum_{i=1}^n x_ip_i$. Choose elements $y_1,\ldots,y_n\in
N$ so that $\|y_i\|\leq 1$ and $\|x_i-y_i\|\leq \gamma'$, and let $y=\sum_{i=1}^n y_ip_i \in {\mathrm{Alg}}(N\cup B)$. Then
$\|y\|\leq 1$ and
\begin{equation}\label{eq1}
\|x-y\|=\|\sum_{i=1}^n (x_i-y_i)p_i\|=\max\{\|x_i-y_i\|: 1\leq i\leq n\}\leq \gamma'.
\end{equation}
The argument is symmetric in $M$ and $N$, so $d(C^*(M\cup B),C^*(N\cup B))\leq \gamma'$. The result follows from the Kaplansky density theorem via \cite[Lemma 5]{KK:AJM}.
\end{proof}


The next lemma takes a tensor product factor $M=P\vnotimes Q$ acting on $\Hs_1\otimes\Hs_2$ and considers close factors $N$ generated by commuting II$_1$ factors $S$ and $T$ which are assumed close to $P$ and $Q$ respectively.  The lemma shows that, provided we have a reverse near containment of $M'$ into $N'$, then we can make a small unitary perturbation of $N$, $S$ and $T$ so that $S$ can be viewed as acting on $\Hs_1$ and $T$ on $\Hs_2$.
\begin{lemma}\label{lemTP.4}
Let $P\subseteq \mathcal B(\Hs_1)$ and $Q\subseteq \mathcal B(\Hs_2)$ be II$_1$ factors, let $\Hs=\Hs_1\otimes \Hs_2$, and let $M=P\vnotimes Q$. Suppose that $N\subseteq \mathcal B(\Hs)$ is a II$_1$ factor and has two commuting subfactors
$S$ and $T$ so that
\begin{equation}\label{eqt.3.1}
d(M,N),\ d(P\otimes I_{\Hs_2},S),\ d(I_{\Hs_1}\otimes Q,T)<\lambda,\quad M'\subset_{k\lambda} N',
\end{equation}
for constants $\lambda,k>0$ satisfying
\begin{equation}\label{NewEq.1}
(90301+27180600k)\lambda<1/100.
\end{equation}
Then there exists a unitary $u\in \mathcal B(\Hs)$ such that
\begin{equation}\label{eqt.3.2}
\|I_\Hs -u\|<150(90602+27271202k)\lambda,
\end{equation}
$u^*Su\subseteq \mathcal B(\Hs_1)\otimes I_{\Hs_2}$, $u^*Tu\subseteq I_{\Hs_1}\otimes\mathcal B(\Hs_2)$, and
\begin{equation}\label{eqt.3.3}
d(M,u^*Nu),\ d(P,u^*Su),\ d(Q,u^*Tu)\leq (27180601+8181360600k)\lambda.
\end{equation}
\end{lemma}
\begin{proof}
Let $A$ be a masa in $(P'\cap \mathcal B(\Hs_1))\otimes I_{\Hs_2}$. Then $A\subseteq M'\subset_{k\lambda}N'$. Since $k\lambda <1/100$ there exists, by the embedding theorem (Theorem \ref{Injective}), a unitary $u_1\in
(A\cup N')''$ such that $u_1Au_1^*\subseteq N'$ and $\|I_\Hs -u_1\|\leq 150k\lambda$. Let $N_1=u_1^*Nu_1$, $S_1=u_1^*Su_1$ and $T_1=u_1^*Tu_1$.
Then $A\subseteq (P\otimes I_{\Hs_2})'\cap S_1'$ and
\begin{equation}\label{eqt.3.4}
d(M,N_1),\ d(P\otimes I_{\Hs_2},S_1),\ d(I_{\Hs_1}\otimes Q,T_1)< (1+300k)\lambda,\quad M'\subset_{301k\lambda}N_1'.
\end{equation}
By Lemma \ref{lemTP.2},
\begin{equation}\label{eqt.3.5}
d(((P\otimes I_{\Hs_2})\cup A)'',(S_1\cup A)'')< (1+300k)\lambda,
\end{equation}
and $((P\otimes I_{\Hs_2})\cup A)''$ is amenable since $((P\otimes I_{\Hs_2})\cup A)'=(A\cup (I_{\Hs_1}\otimes \mathcal B(\Hs_2)))''$, which is amenable. By the embedding theorem (Theorem \ref{Injective}) there is a unitary $u_2\in (((P\otimes I_{\Hs_2})\cup A)''\cup (S_1\cup A)'')''$ such
that $u_2((P\otimes I_{\Hs_2})\cup A)''u_2^*=(S_1\cup A)''$ and $\|I_\Hs-u_2\|\leq 150(1+300k)\lambda$. Let $N_2=u_2^*N_1u_2$, $S_2=u_2^*S_1u_2$ and
$T_2=u_2^*T_1u_2$. Then
\begin{equation}\label{eqt.3.6}
d(M,N_2),\ d(P\otimes I_{\Hs_2},S_2),\ d(I_{\Hs_1}\otimes Q,T_2)\leq 301(1+300k)\lambda,\quad M'\subset_{(300+90301k)\lambda}N_2'.
\end{equation}
Moreover,
\begin{equation}\label{eqt.3.7}
S_2\subseteq u_2^*(S_1\cup A)''u_2=((P\otimes I_{\Hs_2})\cup A)''\subseteq \mathcal B(\Hs_1)\otimes I_{\Hs_2}.
\end{equation}

Now choose a masa $B\subseteq I_{\Hs_1}\otimes (Q'\cap \mathcal B(\Hs_2))$. Then $B\subseteq M'\subset_{(300+90301k)\lambda}N_2'$. The estimate (\ref{NewEq.1}) allows the embedding theorem (Theorem \ref{Injective}) to be applied to give a unitary $u_3\in
(B\cup N_2')''$ so that $u_3Bu_3^*\subseteq N_2'$ and $\|I_\Hs-u_3\|\leq 150(300+90301k)\lambda$. Let $N_3=u_3^*N_2u_3$, and $T_3=u_3^*T_2u_3$, and note that $S_2=u_3^*S_2u_3$ since $S_2$ commutes with $B$ and $N_2'$.
We also have the estimates
\begin{equation}\label{eqt3.8}
d(M,N_3),\ d(Q,T_3)\leq (90301+27180600k)\lambda.
\end{equation}
By construction, $B\subseteq (I_{\Hs_1}\otimes Q)'\cap T_3'$, so by Lemma \ref{lemTP.2} and the inequality \eqref{NewEq.1},
\begin{equation}\label{eqt3.9}
d(((I_{\Hs_1}\otimes Q)\cup B)'', (T_3\cup B)'')\leq (90301+27180600k)\lambda<1/100.
\end{equation}

As $((I_{\Hs_1}\otimes Q)\cup B)''$ is amenable (it is the commutant of the amenable algebra $((\mathcal B(\Hs_1)\otimes I_{\Hs_2})\cup B)''$), Theorem \ref{Injective} gives a unitary $u_4\in (((I_{\Hs_1}\otimes Q)\cup B)''\cup (T_3\cup B)'')''$ with the property that
$u_4((I_{\Hs_1}\otimes Q)\cup B)''u_4^*=(T_3\cup B)''$ and $\|I_\Hs-u_4\|\leq 150(90301+27180600k)\lambda$. Since $S_2$ commutes with 
$I_{\Hs_1}\otimes Q$, $B$ and $T_3$, we see that
$u_4^*S_2u_4=S_2$. Also 
\begin{equation}\label{e2.12}
u_4^*T_3u_4\subseteq ((I_{\Hs_1}\otimes Q)\cup B)''\subseteq I_{\Hs_1}\otimes \mathcal B(\Hs_2).
\end{equation}
 Consequently the desired unitary $u$ is $u_1u_2u_3u_4$, and
\begin{equation}\label{eqt.3.10}
\|I_\Hs-u\|\leq \sum_{i=1}^4 \|I_\Hs-u_i\|\leq 150(90602+27271202k)\lambda
\end{equation}
from previous estimates, while
\begin{equation}\label{eqt.3.11}
d(M,u^*Nu),\ d(P\otimes I_{\Hs_2},u^*Su),\ d(I_{\Hs_1}\otimes Q,u^*Tu)\leq (27180601+8181360600k)\lambda.
\end{equation}
We have $u^*Su\subseteq \mathcal B(\Hs_1)\otimes 1_{\Hs_2}$ from (\ref{eqt.3.7}) as $u_3$ and $u_4$ commute with $S_2=u_2^*u_1^*Su_1u_2$ and $u^*Tu\subseteq I_{\Hs_1}\otimes\mathcal B(\Hs_2)$ from (\ref{e2.12}).
\end{proof}

Lemma 2.16 of \cite{CCSSWW:Duke} considers near containments of relative commutants.  We will use the following version of this lemma in the context of distance estimates.  The proof is identical to part (i) of \cite[Lemma 2.16]{CCSSWW:Duke}, noting that if the $y\in N$ in the proof of that lemma lies in the unit ball, then this is also true for the approximating elements $E_{Q'\cap N}(y)$. 
\begin{lemma}[{c.f. \cite[Lemma 2.16(i)]{CCSSWW:Duke}}]\label{NewLem}
Let $M$ and $N$ be II$_1$ factors acting nondegenerately on a Hilbert space and suppose that $P\subseteq M$ and $Q\subseteq N$ are unital von Neumann subalgebras.  Then $d(P'\cap M,Q'\cap N)\leq d(M,N)+2\sqrt{2}d(P,Q)$.
\end{lemma}

\begin{lemma}\label{lemTP.3}
Suppose that $M$ and $N$ are II$_1$ factors acting nondegenerately on a Hilbert space $\Hs$ and that $\dim_M \Hs =1$. If $d(M,N) < 1/(301\times 136209)=1/40998909$, then
$\dim_N\Hs =1$.
\end{lemma}
\begin{proof}
Choose $\gamma$ to satisfy $d(M,N)<\gamma<1/40998909$ and choose a masa $A\subseteq M$. Then $A\subseteq_\gamma N$ and $\gamma< 1/100$, so by the embedding theorem (Theorem \ref{Injective}),
there exists a unitary $u\in (A\cup N)''$ with $\|u-I_\Hs\|\leq 150\gamma$ so that $uAu^*\subseteq N$. Let $N_1=u^*Nu$, so that $d(M,N_1)\leq 301\gamma<1/136209$. Since
$A\subseteq M\cap N_1$, we may apply \cite[Proposition 4.6]{CCSSWW:Duke} to $M$ and $N_1$ to conclude that $\dim_{N_1}\Hs =1$. Since $N$ is unitarily conjugate to $N_1$, it
follows that $\dim_N\Hs =1$ as required.
\end{proof}

We now turn to  the tensor product decomposition in a II$_1$ factor $N$ close to a tensor product $M\cong P\vnotimes Q$, using the reduction to standard form technique of \cite[Section 4]{CCSSWW:Duke}. We do this first under the assumption that both factors $M$ and $N$ contain a suitable hyperfinite subfactor; this assumption is removed in the subsequent theorem by means of the embedding theorem.

\begin{lemma}\label{gen}
Let $M$ and $N$ be II$_1$ factors  acting nondegenerately on a  Hilbert space $\Hs$ with $d(M,N)<\gamma$.  Suppose that $M$ is generated by two commuting II$_1$ factors $P$ and $Q$ and  that there are hyperfinite II$_1$ factors $R_1\subseteq P$ and $R_2\subseteq Q$ with $R_1'\cap P=\mathbb CI_P$ and $R_2'\cap Q=\mathbb CI_Q$ which further satisfy $(R_1\cup R_2)''\subseteq N$.  Write $S=R_2'\cap N$ and $T=R_1'\cap N$.  Then the following statements hold:
\begin{enumerate}[(i)]
\item $d(P,S)<\gamma$ and $d(Q,T)<\gamma$;\label{NewLem.Part1}
\item if $\gamma<\frac{1}{2\sqrt{2}+2}$, then $S'\cap N=T$ and $T'\cap N=S$;\label{NewLem.Part2}
\item if $\gamma<10^{-39}$, then $N$ is generated by the commuting II$_1$ factors $S$ and $T$.\label{NewLem.Part3}
\end{enumerate}
\end{lemma}
\begin{proof}
By \cite{NT:PJA} we may view $M$ as $P\,\overline{\otimes}\, Q$, and it follows from Tomita's commutation theorem (see \cite[Theorem IV.5.9 and Corollary IV.5.10]{Tak1}) that $P=R_2'\cap M$. Similarly, $Q=R_1'\cap M$. Part (\ref{NewLem.Part1}) then follows from Lemma 3.3.

For (\ref{NewLem.Part2}), note that $R_1\subseteq S$, so that $S'\cap N\subseteq R_1'\cap N=T$. Applying Lemma \ref{NewLem} to the close pairs $(M,N)$ and $(S,P)$ gives 
\begin{equation}\label{gen.1}
Q=P'\cap M\subseteq_{(2\sqrt{2}+1)\gamma} S'\cap N\subseteq T.
\end{equation}
Since $d(T,Q)<\gamma$, it follows from \eqref{gen.1} that 
\begin{equation}
T\subset_{(2\sqrt{2}+2)\gamma}S'\cap N\subseteq T.
\end{equation}
By hypothesis,  $(2\sqrt{2}+2)\gamma<1$ and this  ensures that $S'\cap N=T$ (see \cite[Proposition 2.4]{CSSWW:Acta}).  The identity $T'\cap N=S$ is obtained similarly.

Now we turn to (\ref{NewLem.Part3}). Since $\gamma <1/87$, \cite[Lemma 4.8]{CCSSWW:Duke} gives an integer $n$, a nonzero projection $e\in M'$, 
and a unitary $u\in (M'\cup N')''$ such that
$e\in (u^*Nu)'$,
\begin{equation}\label{eq3.1}
\|u-I_\Hs\|\leq 12\sqrt{2}(1+\sqrt{2})\gamma+4\sqrt{2}((1+\sqrt{2})\gamma)^{1/2},
\end{equation}
and $\dim_{Me}(e\Hs)=1/n$. Let $\Ks=(e\Hs)\otimes {\mathbb{C}}^n$, and define factors by $M_1=(Me)\otimes I_{{\mathbb{C}}^n}$, $P_1=(Pe)\otimes I_{{\mathbb{C}}^n}$, $Q_1=(Qe)\otimes I_{{\mathbb{C}}^n}$, $R_3=({R}_1e)\otimes I_{{\mathbb{C}}^n}$, $R_4=({R}_2e)\otimes I_{{\mathbb{C}}^n}$, and also let $N_1=((u^*Nu)e)\otimes I_{{\mathbb{C}}^n}$, $S_1=((u^*Su)e)\otimes I_{{\mathbb{C}}^n}$, $T_1=((u^*Tu)e)\otimes I_{{\mathbb{C}}^n}$. Then
$M_1$ and $N_1$ are faithful normal representations of $M$ and $N$ respectively on $\Ks$, and $(R_3\cup R_4)''\subseteq N_1$ since $u$ commutes with $M\cap N$. 

Combining \eqref{eq3.1} and the inequality $\gamma<10^{-19}\gamma^{1/2}$, we have the estimate
\begin{equation}\label{eq3.2}
d(M_1,N_1)\leq \gamma +2\|u-I_\Hs\|\leq (49+24\sqrt{2})\gamma+8\sqrt{2}((1+\sqrt{2})\gamma)^{1/2}<18\gamma^{1/2}.
\end{equation}
Let $\gamma_1$ denote the last term in \eqref{eq3.2}, so that $d(M_1,N_1)\leq \gamma_1$. By construction, $\dim_{M_1} \Ks =1$ so $M_1$ is in standard position on $\Ks$. As the bound on $\gamma$ ensures that $\gamma_1<1/40998909$, Lemma \ref{lemTP.3} shows that $N_1$ is also in standard position on $\Ks$. If we represent $P_1$ and $Q_1$ in standard position on Hilbert spaces $\Ks_1$ and $\Ks_2$ respectively, then $P_1\vnotimes Q_1\cong M_1$ is in standard position on $\Ks_1\otimes \Ks_2$. This allows us to assume that $\Ks=\Ks_1\otimes \Ks_2$ and to identify $P_1$ with $P_1\otimes I_{\Ks_2}$ and $Q_1$ with $I_{\Ks_1}\otimes Q_1$.  As both $M_1$ and $N_1$ are in standard position, \cite[Lemma 4.1(i)]{CCSSWW:Duke} gives
\begin{equation}\label{eq3.8}
M_1'\subset_{2(1+\sqrt{2})\gamma_1} N_1',\qquad N_1'\subset_{2(1+\sqrt{2})\gamma_1}M_1'.
\end{equation}

 The hypotheses of Lemma \ref{lemTP.4} are now met by taking $k=2(1+\sqrt{2})$ and $\lambda=\gamma_1=18 \gamma^{1/2}$. Thus there exists a unitary $u_1\in \mathcal B(\Ks)$ such
that
\begin{equation}
\|I_\Ks-u_1\|<150(90602+54542404(1+\sqrt{2}))\gamma_1,
\end{equation}
and if we define $N_2=u_1^*N_1u_1$, $S_2=u_1^*S_1u_1$, and $T_2=u_1^*T_1u_1$, then $S_2\subseteq \mathcal B(\Ks_1)$, $T_2\subseteq \mathcal B(\Ks_2)$, and
\begin{align}
d(M_1,N_2),\ d(P_1,S_2),\ d(Q_1,T_2) &<(27180601+16362721200(1+\sqrt{2}))\gamma_1\notag\\&<1/40998909,\label{eqTP.8.1}
\end{align}
from the choice of the bound on $\gamma$. By Lemma \ref{lemTP.3}, $S_2$ is in standard position on $\Ks_1$ and similarly $T_2$ is in standard position on $\Ks_2$.  It follows that $(S_2\cup T_2)''$, which is canonically identified with $S_2\vnotimes  T_2$ with respect to $\Ks=\Ks_1\otimes\Ks_2$, is also in standard position on $\Ks$.   Since $(S_2\cup T_2)''\subseteq N_2$ and $\dim_{N_2}\Ks=1$, we conclude that $(S_2\cup T_2)''=N_2$, and hence also that $(S\cup T)''=N$.
\end{proof}

We are now in a position to show that tensorial decompositions can be transferred between close II$_1$ factors.

\begin{theorem}\label{thmTP.8}
Let $M$ and $N$ be II$_1$ factors with separable preduals, acting nondegenerately on a Hilbert space $\Hs$.  If $M$ is
generated by two commuting II$_1$ factors $P$ and $Q$ and $d(M,N)<\gamma<3.3\times 10^{-42}$, then there exist commuting II$_1$ subfactors $S$ and $T$ which generate $N$ and satisfy
\begin{equation}
d(P,S),\ d(Q,T) < (200\sqrt{2}+1)\gamma<284\gamma,\quad d_{cb}(P,S),\ d_{cb}(Q,T)\leq 601\,\gamma.
\end{equation}
\end{theorem}
\begin{proof}
By \cite{P:Invent}, choose  amenable subfactors $R_1\subseteq P$ and $R_2\subseteq Q$ with trivial relative commutants.  Then $(R_1\cup R_2)''$ is also amenable, and we denote this factor by $R$. Since $R\subset_\gamma N$, we may choose a unitary $v\in (R\cup N)''$ with $\|v-I_\Hs\|\leq 150 \gamma$, $\|x-vxv^*\|\leq 100\gamma\|x\|$ for $x\in R$ and $vRv^*\subseteq N$ by the embedding theorem (Theorem \ref{Injective}). Write $N_1=v^*Nv$ so that $R\subseteq M\cap N_1$ and $d(M,N_1)<\gamma_1=301\gamma$.  Since $301\gamma<10^{-39}$, Lemma \ref{gen} (iii) can be applied to conclude that $N_1$ is generated by the commuting subfactors $S_1=R_2'\cap N_1$ and $T_1=R_1'\cap N_1$. Hence $N$ is generated by the commuting subfactors $S:=(vR_2v^*)'\cap N=vS_1v^*$ and $T:=(vR_1v^*)'\cap N=vT_1v^*$.  Since $d(R_1,vR_1v^*),d(R_2,vR_2v^*)\leq 100\gamma$, Lemma \ref{NewLem} shows that $d(P,S)\leq (200\sqrt{2}+1)\gamma$ and similarly $d(Q,T)\leq (200\sqrt{2}+1)\gamma$.

We now estimate the cb-distance between $P$ and $S$, so fix $n\in \mathbb{N}$ and let $F$ denote a unital subalgebra of $R_2$ isomorphic to a copy of the $n\times n$
matrices $\mathbb{M}_n$. By construction, $F\subseteq R_2\subseteq Q\cap T_1$, so there are induced factorizations $Q \cong F\,\otimes\, Q_0$, $T_1\cong F\,\otimes\,
T_0$ and $R_2\cong F\,\otimes\, R_0$ where $Q_0=F'\cap Q$, $T_0=F'\cap T_1$ and $R_0=F'\cap R_2$. Thus $M$ is generated by the two commuting factors
$P_0=(P\cup F)''$ and $Q_0$ (amounting to taking a copy of $\mathbb{M}_n$ from $Q$ and attaching it to $P$) and $N_1$ by the commuting factors $S_0=(S_1\cup F)''$ and $T_0$. We note that $R_0\subseteq Q_0\cap T_0$ and has trivial relative
commutants in $Q_0$ and $T_0$. In this way $(P\cup F)''=R_0'\cap M$ and $(S_1\cup F)''=R_0'\cap N_1$. Another application of Lemma \ref{NewLem} gives
\begin{equation}\label{saw3.23}
d((S_1\cup F)'',(P\cup F)'')= d(R_0'\cap N_1,R_0'\cap M)\leq d(M,N_1)\leq 301\,\gamma.
\end{equation}

Since $F$ is a factor, there is an isometric $*$-isomorphism between
$(F'\cup F)''\subseteq \mathcal{B}(\Hs)$ and $F'\,\otimes\, F\cong
F'\,\otimes\, \mathbb{M}_n\subseteq
\mathcal{B}(\Hs)\,\otimes\,\mathbb{M}_n$, defined on generators by
$f'f\mapsto f'\otimes f$, which carries $(P\cup F)''$ and $(S_1\cup F)''$
respectively to $P\,\otimes\,\mathbb{M}_n$ and
$S_1\,\otimes\,\mathbb{M}_n$. In this way (\ref{saw3.23}) gives $d_{\cb}(P,S_1)\leq 301\gamma$.  As $S=vS_1v^*$ where $v$ is a unitary satisfying $\|v-I_{\Hs}\|\leq 150\,\gamma$, it follows that $d_{\cb}(S,S_1)\leq 300\gamma$, whence the triangle inequality gives $d_{\cb}(P,S)\leq 601\gamma$. The estimate on $d_{cb}(Q,T)$ is proved in the same way.
\end{proof}

The following corollary is a rewording of the last theorem.

\begin{corollary}\label{corTP.9}
Let $M$ and $N$ be II$_1$ factors with separable preduals, acting nondegenerately on a Hilbert space $\Hs$, and suppose that $d(M,N)< 3.3\times 10^{-42}$. If $M$ is
prime, then so too is $N$.
\end{corollary}

In \cite{CCSSWW:Duke}, we encapsulated the weakest form of the Kadison-Kastler conjecture by defining a II$_1$ factor $M$ to be weakly Kadison-Kastler stable if there exists $\vp>0$ so that if $\pi : M\to \mathcal B(\Hs)$ is a normal representation and $N\subseteq \mathcal B(\Hs)$ is a II$_1$ factor satisfying $d(\pi(M),N)<\vp$, then $\pi(M)$ and $N$ are $*$-isomorphic. Theorem \ref{thmTP.8} shows that this property is preserved under taking tensor products.

\begin{corollary}\label{thmTP.11}
Let $P$ and $Q$ be II$_1$ factors with separable preduals and suppose that both are weakly Kadison-Kastler stable. Then so too is $M:= P\vnotimes Q$.
\end{corollary}
\begin{proof}
Let $\eps>0$ be small enough to satisfy the definition of weak Kadison-Kastler stability of both $P$ and $Q$.
Suppose that $M$ and $N$ are represented on some Hilbert space.  When $d(M,N)$ is sufficiently small, Theorem \ref{thmTP.8} shows that $N$ is generated by two commuting II$_1$ factors $S$ and $T$ such that $d(P,S)<\eps$ and $d(Q,T)<\eps$.  Thus $P\cong S$ and $Q\cong T$, from which it follows that $M\cong P\vnotimes Q\cong S\vnotimes T\cong N$. Hence $M$ is weakly Kadison-Kastler stable.
\end{proof}

We now turn to the strongest form of the Kadison-Kastler conjecture, which asks that close von Neumann algebras arise from small unitary perturbations.  As in \cite{CCSSWW:Duke}, we say that a II$_1$ factor $M$ is \emph{strongly Kadison-Kastler stable} if, given $\vp>0$, there exists $\delta>0$ with the following property: if $\pi : M\to \B(\Hs)$ is a normal representation and $N\subseteq \B(\Hs)$ is a II$_1$ factor satisfying $d(\pi(M),N)<\delta$, then there exists a unitary $u\in \B(\Hs)$ with $\|I_{\Hs} - u\|<
\vp$ such that $u\pi(M)u^*=N$.  We need a standard observation regarding representations of tensor products.

\begin{lemma}\label{lemTP.7}
Let $M$ and $N$ be II$_1$ factors and let $\pi : M\vnotimes N\to \mathcal B(\Hs)$ be a normal
representation on a Hilbert space $\Hs$. Then there exists a type I$_\infty$ factor $P$ such that
\begin{equation}\label{eqt.6.1}
\pi(M\otimes 1_N)\subseteq P\subseteq \pi(1_M\otimes N)'.
\end{equation}
\end{lemma}
\begin{proof}
Let $\lambda$ denote the standard representation of $M\vnotimes N$ on $L^2(M)\otimes L^2(N)$. From the general form of normal representations of von Neumann algebras, there exists a Hilbert space $\Ks$ and a projection $p\in \lambda(M\vnotimes N)'\vnotimes\mathcal B(\Ks)=J_MMJ_M\vnotimes J_NNJ_N\vnotimes\mathcal B(\Ks)$ so that $\pi$ is unitarily equivalent to the representation $p(\lambda(x)\otimes 1_\Ks)$, for $x\in M\vnotimes N$. Since $\lambda(M\vnotimes N)'\vnotimes\mathbb B(\Ks)$ is a type II factor, the projection $p$ is Murray-von Neumann equivalent to $p_1\otimes q_1\otimes f$, where the projections $p_1,q_1$ and $f$ lie in $J_MMJ_M$, $J_NNJ_N$ and $\mathcal B(\Ks)$ respectively.  A partial isometry implementing this equivalence provides a unitary conjugacy between the representations $p\lambda(\cdot)$ and $(p_1\otimes q_1\otimes f)\lambda(\cdot)$, so $\pi$ is unitarily equivalent to $\pi_1=(p_1\otimes q_1\otimes f)\lambda(\cdot)$.  For this latter representation we can verify the statement of the lemma with the I$_\infty$ factor $P_1=p_1(\mathcal B(L^2(M))p_1\vnotimes q_1\vnotimes f\mathcal B(\Ks)f$, and hence an appropriate unitary conjugation provides the required $P$ for the representation $\pi$.
\end{proof}

By the results of \cite{CCSSWW:InPrep}, strong Kadison-Kastler stability for a II$_1$ factor $M$ implies that $M$ has a positive solution to Kadison's similarity problem.  To our knowledge, it is not known whether a tensor product of two II$_1$ factors with the similarity property necessarily also has the similarity property, so to obtain preservation results for strongly Kadison-Kastler stable factors we need to impose an additional hypothesis to take care of the similarity property. In the next lemma, this is the condition that $\pi(M)'\subset_\delta N'$.
\begin{lemma}\label{newtplem}
Let $P$ and $Q$ be strongly Kadison-Kastler stable II$_1$ factors with separable preduals and let $M=P\vnotimes Q$. Then, for all $\eps>0$, there exists $\delta>0$ with the following property: if $\pi:M\rightarrow\mathcal B(\Hs)$ is a normal representation and $N\subseteq\mathcal B(\Hs)$ is a von Neumann algebra with $d(\pi(M),N)<\delta$ and $\pi(M)'\subset_\delta N'$, then there exists a unitary $u\in\mathcal B(\Hs)$ with $\|u-I_\Hs\|<\eps$ such that $u\pi(M)u^*=N$.
\end{lemma}
\begin{proof}
Fix $\vp<1/50$. If we apply the strong stability hypothesis to $P$ and $Q$ with $\vp$ replaced by $\vp/4$, then there exists $\delta_0>0$ with the following property: if
$\sigma:P\to\mathcal B(\Hs)$ is a normal representation and $S\subseteq \mathcal B(\Hs)$ satisfies $d(\sigma(P),S)<\delta_0$, then $\sigma(P)$ and $S$ are unitarily conjugate by a unitary
$u\in \mathcal B(\Hs)$ satisfying $\|I_\Hs-u\|<\vp/4$, with a similar statement for $Q$. Now choose $\delta >0$ so small that the following three inequalities are satisfied:
\begin{equation}\label{eqt.8.1}
\delta<3.3\times 10^{-42},
\end{equation}
\begin{equation}\label{eqt.8.2}
 150(90602+27271202)\times 284\,\delta<\vp/2<1/100,
\end{equation}
\begin{equation}\label{eqt.8.3}
(27180601+8181360600)\times 284\,\delta<\delta_0.
\end{equation}

Let $\pi:M\rightarrow \mathcal B(\Hs)$ be a normal representation and let $N\subseteq \mathcal B(\Hs)$ be such that $d(\pi(M),N)<\delta$. Let us write $M_1=\pi(M)$, $P_1=\pi(P\otimes I_Q)$ and $Q_1=\pi(I_P\otimes Q)$. Since $\delta <3.3\times 10^{-42}$, Theorem \ref{thmTP.8} shows that $N$ is generated by two commuting subfactors $S$ and $T$ satisfying
\begin{equation}
d(P_1,S),\ d(Q_1,T)<\delta_1:= 284\,\delta.
\end{equation}
By Lemma \ref{lemTP.7}, there is a type I$_\infty$ factor lying between $P_1$ and $Q_1'$. By \cite[Theorem 9.3.2]{KRin:Book2} there is, up to unitary equivalence, a decomposition
of $\Hs$ as $\Hs_1\otimes \Hs_2$ such that this type I$_\infty$ factor is $\mathcal B(\Hs_1)\otimes I_{\Hs_2}$, whereupon $P_1\subseteq \mathcal B(\Hs_1)\otimes I_{\Hs_2}$ and $Q_1\subseteq I_{\Hs_1}\otimes \mathcal B(\Hs_2)$. From \eqref{eqt.8.2}, the inequalities in the hypotheses of Lemma \ref{lemTP.4}
are satisfied for $\lambda=\delta_1$ and $k=1$, so by that lemma there exists a unitary $u_1\in \mathcal B(\Hs)$ with the following properties. The inequality
\begin{equation}\label{eqt.8.4a}
\|I_\Hs-u_1\|<150(90602+27271202)\delta_1
\end{equation}
holds and, upon setting $N_1=u_1^*Nu_1$, $S_1=u_1^*Su_1$, $T_1=u_1^*Tu_1$, we have $S_1\subseteq \mathcal B(\Hs_1)\otimes I_{\Hs_2}$ and $T_1\subseteq I_{\Hs_1}\otimes \mathcal B(\Hs_2)$. Moreover,  the estimates
\begin{equation}
d(M_1,N_1),\ d(P_1,S_1),\ d(Q_1,T_1)\leq (27180601+8181360600)\delta_1<\delta_0
\end{equation}
are valid, where the last inequality is \eqref{eqt.8.3}.  Thus there exist unitaries $v\in \mathcal B(\Hs_1)$ and $w\in \mathcal B(\Hs_2)$ such that $(v\otimes I_{\Hs_2})P_1(v\otimes I_{\Hs_2})^*=S_1$, $(I_{\Hs_1}\otimes w)Q_1(I_{\Hs_1}\otimes w)^*=T_1$
and the inequalities $\|I_{\Hs_1}-v\|,\ \|I_{\Hs_2}-w\|<\vp/4$ hold. Let $u_2=v\otimes w$ and observe that $\|I_\Hs-u_2\|<\vp/2$. If we define $u=u_2u_1$, then $u\pi(M)u^*=N$ and
\begin{align}
\|I_\Hs-u\|&\leq \|I_\Hs-u_1\|+\|I_\Hs-u_2\|\notag\\
&<150(90602+27271202)\delta_1+\vp/2<\vp
\end{align}
from \eqref{eqt.8.4a} and \eqref{eqt.8.2}.
\end{proof}

The most general class of II$_1$ factors known to have the similarity property (\cite{C:JOT,Pi:IJM,C:JFA}) are those with Murray and von Neumann's property $\Gamma$.  By definition, property $\Gamma$ passes to tensor products, yielding the following result.
\begin{theorem}\label{thmTP.12}
Let $P$ and $Q$ be II$_1$ factors with separable preduals and suppose that both are strongly Kadison-Kastler stable. Suppose further that at least one has property $\Gamma$. Then $M:= P\vnotimes Q$ is strongly Kadison-Kastler stable.
\end{theorem}
\begin{proof}
Let $\pi:M\rightarrow \mathcal B(\Hs)$ be a faithful normal representation.  Let $N\subseteq \mathcal B(\Hs)$ be another II$_1$ factor with $d(\pi(M),N)<1/190$ so that $N$ inherits property $\Gamma$ from $M=P\vnotimes Q$ from Proposition \ref{GammaProp}.   Proposition 2.4(ii) of \cite{CCSSWW:Duke} shows that if $d(\pi(M),N)<\gamma$, then the near inclusion  $\pi(M)'\subset_{5\gamma}N'$ holds. Strong Kadison-Kastler stability now follows from Lemma \ref{newtplem}.
\end{proof}

\begin{remark}\label{remTP.13}
 The examples of strongly Kadison-Kastler stable II$_1$ factors constructed in \cite{CCSSWW:Duke}  all have the form $(P\rtimes_\alpha G)\vnotimes R$ where $P$ is amenable and $G$ is $SL_n(\mathbb{Z})$ for $n\geq 3$, and these have property $\Gamma$ since they are McDuff. These satisfy the hypothesis of Theorem \ref{thmTP.12} and thus new examples of strongly Kadison-Kastler stable factors can be generated by taking finite tensor products of the existing ones
 of \cite{CCSSWW:Duke}.$\hfill\square$
\end{remark}

\section{Open questions}\label{S4}

We end with a short list of open problems.

\begin{enumerate}
\item Does property (T) transfer to sufficiently close subalgebras?
\item What can be said about the fundamental group, or outer automorphism group of close II$_1$ factors?
\item How do nonamenable subalgebras of close II$_1$ factors, such as subfactors behave?  If $M$ and $N$ are sufficiently close II$_1$ factors and $M$ has an index $2$ subfactor, must $N$ also have an index $2$ subfactor?
\item Does a non-prime II$_1$ factor have the similarity property?  Less generally, does the tensor product of two II$_1$ factors with the similarity property have the similarity property?
\end{enumerate}

\subsection*{Acknowledgements} SW thanks Ionut Chifan for useful conversations about von Neumann algebras close to tensor products which sparked the line of research developed in section \ref{S3}.

\end{document}